\documentclass[reqno]{amsart}
\usepackage{amssymb}
\usepackage[left=3cm,right=3cm,top=0cm,bottom=2cm]{geometry} 
\usepackage{amsmath} 
 
\usepackage{hyperref}
\usepackage[lofdepth,lotdepth]{subfig}

\newtheorem{theorem}{Theorem}[section]
\newtheorem{lemma}[theorem]{Lemma}

\theoremstyle{definition}
\newtheorem{definition}[theorem]{Definition}
\newtheorem{exemple}[theorem]{Example}

\theoremstyle{remark}
\newtheorem{remark}[theorem]{Remark}

\numberwithin{equation}{section}

\begin{document}
\title{Fixed point theorems in controlled rectangular metric spaces}
\author[Mohamed Rossafi and Abdelkarim Kari]{Mohamed Rossafi$^{1*}$ {and} Abdelkarim Kari$^{2}$}

\address{$^{1}$LaSMA Laboratory Department of Mathematics, Faculty of Sciences Dhar El Mahraz, University Sidi Mohamed Ben Abdellah, B. P. 1796 Fes Atlas, Morocco}
\email{rossafimohamed@gmail.com; mohamed.rossafi@usmba.ac.ma}
\address{$^{2}$AMS Laboratory Faculty of Sciences Ben M’Sik, Hassan II University, Casablanca, Morocco}
\email{abdkrimkariprofes@gmail.com}
\date{
	\newline \indent $^{*}$Corresponding author}
\subjclass[2010]{Primary 47H10; Secondary 54H25.}
\keywords{Fixed point, controlled metric type space, extended rectangular b-metric space.}
\begin{abstract}
In this paper, we introduce an extension of rectangular metric spaces called controlled rectangular metric spaces, by changing the rectangular inequality as follows:
\begin{equation*}
	d(x, y)\leq\alpha(x, u)d(x, u)+\alpha(u, v)d(u, v)+\alpha(v, y)d(v, y),
\end{equation*}
for all distinct $x, y, u, v\in X$ with the function $\alpha: X\times X\rightarrow[1,\infty[$. We also establish some fixed point theorems for self-mappings defined on such spaces. Our main results extends and improves many results existing in the literature. Moreover, an illustrative example is presented to support the obtained results.
 \end{abstract}
\maketitle
\section{Introduction}
By a contraction on a metric space $(X, d)$, we understand a mapping $T: X \to X$ satisfying for all $x, y\in X$: $d(Tx, Ty)\leq kd(x, y)$, where $k$ is a real in $[0, 1)$.

In $1922$ Banach proved the following theorem.
\begin{theorem}\cite{BAN}
		Let $(X, d)$ be a complete metric space. Let $T : X \to X$ be a contraction. Then:
	\begin{itemize}
		\item [(i)] $T$ has a unique fixed point $x\in X$.
		\item [(ii)] For every $x_{0}\in X$, the sequence $(x_{n})$, where $x_{n+1} = Tx_{n}$, converges to $x$.
		\item [(iii)] We have the following estimate: For every $x\in X$, $d(x_{n}, x)\leq \dfrac{k^{n}}{1-k}d(x_{0}, x_{1})$, $n\in\mathbb{N}$.
	\end{itemize}
\end{theorem}
Several authors generalise the previous theorem in various directions \cite{BAK, CZ, GEO, kari44, kari22, kari33, KARO, KIR, MLK, RO, RD}.

In $2000$ Branciari  initiated the notion of rectangular metric space.
\begin{definition}
\cite{BRA}. Let $X$ be a non-empty set and $\ d:X\times X\rightarrow \mathbb{R}^{+}$ be a mapping such that for all $x,y$ $\in X$ and for all distinct points $u,v\in X$, each of them different from $x$ and $y,$ on has
\begin{itemize}
\item[(i)]  $d(x,y)=0$ if and only if $x=y$;
\item[(ii)]  $d(x,y)=d(y,x)$  for all distinct points $x,y\in X$; 
\item[(iii)]  $d(x,y)\leq d(x,u)+d(u,v)+d(v,y)$( the rectangular inequality).
\end{itemize}
Then $\left( X,d\right) $ is called an rectangular metric space.
\end{definition}
In $2019$, Z. Asim et al. introduce the concept of extended b-rectangular metric spaces. 
\begin{definition}
\cite{AZ}. Let $X$ be a non empty set, and $ \theta: X\times X\rightarrow \left[ 1,+\infty \right[ $. An extended b-rectangular metric is a function  $ d: X\times X\rightarrow \left[ 0,+\infty \right[ $
 such that for all $x,y$ $\in X$ and all distinct points $u,v\in
X,$ each of them distinct from $x$ and $y$ one has the following conditions:
\begin{itemize}
\item[($ d_1) $] $d\left( x,y\right) =0,$ if and only $x=y;$
\item[($ d_2) $] $d\left( x,y\right) =d\left( y,x\right) ;$
\item[($ d_3) $]$d\left( x,y\right) \leq\theta\left( x,y\right)\left[ d(x,u)+d(u,v)+d(v,y)\right] $.
\end{itemize}
Then $\left( X,d\right) $ is called a  extended rectangular b-metric space. 
\end{definition}
\begin{exemple}
\cite{AZ} Consider $ X=\lbrace	1,2,3,4,5 \rbrace$. Define $ \theta: X\times X\rightarrow \left[ 1,+\infty \right[ $ by
\begin{equation*}
\theta(x,y)=x+y+1\   \forall x,y\in X.
\end{equation*}
Define 
\begin{equation*}
d:X\times X\rightarrow \left[ 0,+\infty \right[ \ by
\end{equation*}
\begin{itemize}
 \item  $ d(x,x) =0 $ for all $ x,y\in X $;
  \item $ d(x,y) =d(y,x)  $ for all $ x,y\in X $;
 \item  $ d(1,3) =d(2,5)=70 , d(1,4) =1000$ and $d(1,5) =1200;$
  \item   $ d(1,2) =d(2,3)=d(3,4)=60, d(3,5)=d(4,5)=d(2,4)=400 $.
   \end{itemize}
Clearly, $ (X,d) $ is extended rectangular b-metric space.
\end{exemple}
In the next section, we introduce controlled rectangular metric space and established some fixed point results for such mappings in complete metric spaces and generalized the results of Kannan \cite{KAN}, Reitch \cite{RE} and Fisher \cite{FICH}.
\section{Main result}
We begin with the following definition.
\begin{definition}
Given a non-empty set $ X $ and $ \alpha: X\times X\rightarrow \left[ 1,+\infty \right[ $.

 The function  $d:X\times X\rightarrow \left[ 0,+\infty \right[ $ is called a controlled rectangular metric if 
\begin{itemize}
\item[($ d_1) $] $d\left( x,y\right) =0,$ if and only $x=y;$
\item[$ (d_2 $)] $d\left( x,y\right) =d\left( y,x\right) ;$
\item[$ (d_3 $)] $ d\left( x,y\right)$ $\leq \alpha\left( x,u\right)$ $d\left( x,u\right) +\alpha\left( u,v\right) d\left( u,v\right) +\alpha\left( v,y\right) d\left( v,y\right).$
\end{itemize}
Then $\left( X,d\right) $ is called a controlled rectangular metric space.  
\end{definition}
\begin{remark}
If, for all $ x,y\in X $, $\alpha\left( x,y\right)=s\geq 1$, then $ (X,d )$ is a rectangular b-metric space, which leads us to conclude that every rectangular b-metric space is a controlled rectangular metric space. Also a controlled rectangular metric space  is not general an extended rectangular b-metric space.
\end{remark}
\begin{exemple}
Consider  $ X= \lbrace1,2,3,4\rbrace$. Define $d:X\times X\rightarrow \left[ 0,+\infty \right[ $ by
\begin{itemize}
\item $ d(x,x) =0 $ for all $ x,y\in X $;
\item $ d(1,2) =\frac {1 }{2} $, $ d(1,3) =\frac {1 }{9} $, $ d(1,4) =\frac {1 }{16} $;
\item $ d(2,3) =\frac {1 }{12} $, $ d(2,4) =\frac {1 }{36} $, $ d(3,4) =\frac {1 }{49} $.
 \end{itemize}
Define $ \alpha: X\times X\rightarrow \left[ 1,+\infty \right[ $ by
\begin{equation*}
\alpha\left( x,y\right) =\max \left\{x,y\right\}, \forall x,y\in X . 
\end{equation*}
Thus, $ (d_1) $ and $ (d_2) $ are clearly true. We shale prove that $ (d_3) $ hold.
We have
\begin{equation*} 
d(1,2)=\frac {1 }{2} \leq\alpha\left( 1,3\right) d\left( 1,3\right) +\alpha\left( 3,4\right) d\left(3,4\right) +\alpha\left( 4,2\right) d\left( 4,2\right)=0.58,
\end{equation*}
and
\begin{equation*} 
 d(1,2)=\frac {1 }{2} \leq\alpha\left( 1,4\right)d\left( 1,4\right) +\alpha\left( 4,3\right) d\left(4,3\right) +\alpha\left( 3,2\right) d\left( 3,2\right)=0.52,
\end{equation*}
Similarly, other cases can be argued. Thus, for all $ x,y\in X $ with distinct $ u,v\in X $. 
We get,
\begin{equation*} d\left( x,y\right) \leq \alpha\left( x,u\right) d\left( x,u\right) +\alpha\left( u,v\right) d\left( u,v\right) +\alpha\left( v,y\right) d\left( v,y\right).
\end{equation*}
Hence, $ (X,d) $ is controlled rectangular metric space.
Note that 
\begin{equation*} 
 d(1,2)=\frac {1 }{2} >0.31=\alpha\left( 1,2\right)\left[ d\left( 1,3\right) + d\left(3,4\right) + d\left( 4,2\right)\right], 
\end{equation*}
 that is, $ d $ is not an extended rectangular b-metric for the $ \alpha=\theta $.
\end{exemple}
On the other hand, in example $ 2.4 $, we replaced $ d(3,4)=60 $ by $ d(3,4)=49 $, we prove that $ d$ is a extended rectangular b-metric.
Note that
\begin{equation*} 
 d(1,4)=1000 >992=\alpha\left( 1,2\right)d(1.2)+\alpha\left( 2,3\right)d(2.3) +\alpha\left(3,4\right)d(3.4). 
\end{equation*}
That is, $d$ is not a controlled  rectangular metric for the $ \theta=\alpha $.

We define Cauchy and convergent sequences in controlled rectangular metric space as follows:
\begin{definition}
Let $\left( X,d\right) $ be a controlled rectangular metric space. $\lbrace x_{n}\rbrace$ be a sequence in X and $x\in X$. The sequence $\lbrace x_{n}\rbrace$ converges to $x$ if and only if $\lim\limits_{n \rightarrow +\infty}d\left( x_{n},x\right)=0. $
\end{definition}
\begin{definition}
Let $(X,d)$ be a controlled rectangular metric space and $\lbrace x_{n}\rbrace$ be a sequence in $X$. We say that
\begin{itemize}	
\item[(i)] $\lbrace x_{n}\rbrace$ is a Cauchy if and only if for every $\varepsilon $ $>0$ there exists a positive integer $N=N(\varepsilon )$ such that $d(x_{n},x_{m})<\varepsilon $, for all $n, m\geq N.$
\item[(ii)] $(X,d)$ is complete if and only if each Cauchy sequence in $X$ is
convergent.
\end{itemize}
\end{definition}
\begin{lemma}
$\ $Let $\left( X,d\right) $ be a controlled rectangular metric space
 and $\lbrace x_{n}\rbrace$ a  Cauchy sequence in X. If $\lbrace x_{n}\rbrace$  converges to $x\in X$ and converges to $y\in X,$
 we assume that
 \begin{equation*}
\lim\limits_{n \rightarrow +\infty}\alpha(x_{n},x),\  \lim\limits_{n \rightarrow +\infty}\alpha(x,x_{n}) \ and  \  \lim\limits_{n,m \rightarrow +\infty}\alpha(x_{n},x_{m}) \  exist \ and \ are \ finite \  \forall n,m\in\mathbb{N},\ n\neq m.
\end{equation*}
then $x=y.$
\end{lemma}
\begin{proof}
If  $\lbrace x_{n}\rbrace$ a Cauchy sequence in $ X $ has two limit point $ x,y\in X $, such that
 $$\lim\limits_{n \rightarrow +\infty}d\left( x_{n},x\right)=0 \ and \ \lim\limits_{n \rightarrow +\infty}d\left( x_{n},y\right)=0 .$$
 Since, $\lbrace x_{n}\rbrace$ is Cauchy, then so from $ (d_3 )$ we have
\begin{equation*}
 d\left( x,y\right) \leq \alpha\left( x,x_{n}\right) d\left( x,x_{n}\right) +\alpha\left( x_{n},x_{n+1}\right) d\left(x_{n},x_{n+1}\right) +\alpha\left( x_{n+1},y\right) d\left(x_{n+1},y\right).
\end{equation*}
By letting $n\rightarrow \infty $ in above inequality, we obtain $  d\left( x,y\right) \leq0 $, which implies that $d(x,y)=0.$
Therefore, $x=y.$  
\end{proof}
\begin{theorem}
Let $ (X,d) $ be a controlled rectangular metric space, and $ T $ a self mapping on $ X $. If there exists $k\in \left] 0,1\right[  $ such that  
\begin{equation}
d(Ty,Tx) >0\Rightarrow(Ty,Tx)\leq k d(x,y)
\end{equation}
For $ x_{0} \in X$, take $ x_{n}=T^{n}x_{0}  $. Suppose that
\begin{equation}
\lim_{i\rightarrow \infty }\sup_{m\geq 1}\alpha\left( x_{i+1}, x_{m}\right) \frac{\alpha\left( x_{i+1}, x_{i+2}\right)+\alpha\left( x_{i+2}, x_{i+3}\right)}{\alpha\left( x_{i}, x_{i+1}\right)+\alpha\left( x_{i+1}, x_{i+2}\right)}<\frac{1}{k^2}.
\end{equation}
We assume that, for $ x\in X $, we have 
\begin{equation}
\lim\limits_{n \rightarrow +\infty}\alpha(x_{n},x),\  \lim\limits_{n \rightarrow +\infty}\alpha(x,x_{n}) \ and  \  \lim\limits_{n,m \rightarrow +\infty}\alpha(x_{n},x_{m}) \  exist \ and \ are \ finite \  \forall n,m\in\mathbb{N},\ n\neq m.
\end{equation}
Then T has a unique fixed point in X.
\end{theorem}
\begin{proof}
We choose any $ x_0 $ be arbitrary, define the iterative sequence $  \lbrace x_n\rbrace$ by
$$ x_1=Tx_0,\ x_2=Tx_1..., x_n=T^{ n}x_0 .$$
 step 1. We shall prove that
\begin{equation*}
\lim_{n\rightarrow \infty }d\left( x_{n},x_{n+1}\right) =0.
\end{equation*}
Now, by the hypothesis of theorem, we have
\begin{align*}
d\left( x_{n},x_{n+1}\right)&= d(T x_{n-1},T x_{n})\\
&\leq k d( x_{n-1}, x_{n})\\
& \leq k^{2} d( x_{n-2}, x_{n-1})\\
& \leq ...\\
&\leq k^{n} d( x_{0}, x_{1})
\end{align*}
 Taking the limit of the above inequality as $ n\rightarrow \infty $, we deduce that 
\begin{equation}
\lim_{n\rightarrow \infty } d\left( x_{n}, x_{n+1}\right)=0.
\end{equation}
 step 2. We shall prove that 
\begin{equation*}
\lim_{n\rightarrow \infty }d\left( x_{n},x_{n+2}\right) =0.
\end{equation*}
We assume that $x_{n}\neq x_{m}$ for every $n,m\in $ $\mathbb{N}.$
Indeed, suppose that $x_{n}= x_{m}$ for some $  n=m+k$  with $ k>0 $, so we have $ Tx_{n}=Tx_{m} $.
And
\begin{equation*}
d( x_{m},x_{m+1})=d( x_{n}, x_{n+1})=d( Tx_{n-1}, Tx_{n})\leq k d( x_{n-1}, x_{n}).
\end{equation*}
Since $k\in  \left] 0,1\right[$. Therefore,    
\begin{equation*}
d( x_{m}, x_{m+1})=d( x_{n}, x_{n+1})< d( x_{n-1}, x_{n}).
\end{equation*}
Continuing this process, we have
\begin{equation*}
d(x_{m}, x_{m+1})< d( x_{m}, x_{m+1}).
\end{equation*}
which is a contradiction. Therefore,
\begin{equation*} 
d( x_{m}, x_{n})>0 \ for\ every\ n,m\in \mathbb{N}, n\neq m 
\end{equation*}
Now, substituting $x=x_{n-1}$ and $y=x_{n+1}$ in the $ (2.1) $, we obtain
\begin{align*}
d( x_{n}, x_{n+2})&=d(T x_{n-1},T x_{n+1})\\
&\leq k d( x_{n-1}, x_{n+1})\\
& \leq k^{2} d( x_{n-2}, x_{n})\\
& \leq ...\\
&\leq k^{n} d( x_{0}, x_{2})
\end{align*} 
 If we take the limit of the above inequality as $ n\rightarrow \infty $ we deduce that 
\begin{equation}
\lim_{n\rightarrow \infty } d\left( x_{n}, x_{n+2}\right)=0.
\end{equation}
 step 3. We shall prove that, $x_{n}$ is a Cauchy sequence in $(X,d)$ i.e, 
\begin{equation*}
\lim_{n,m\rightarrow \infty }d\left( x_{n},x_{m}\right) =0\text{ \ }\forall
n,m\in \mathbb{N}.
\end{equation*}
Denote by $ d_{i}=d(x_{i}, x_{i+1})$ for all $ i\in\mathbf{N} $. We distinguish two cases.

Case 1: Assume that $ m=n+2l+1 $ with $ l\geq 1 $. By property $ (3) $ of the controlled rectangular metric spaces, we have
\begin{align*} 
d( x_{n},x_{m})&=d( x_{n},x_{n+2l+1})\\
&\leq \alpha(x_{n}, x_{n+1})d( x_{n}, x_{n+1})+\alpha(x_{n+1}, x_{n+2})d( x_{n+1}, x_{n+2})+\alpha(x_{n+2}, x_{n+2l+1})d( x_{n+2}, x_{n+2l+1})\\
&\leq \alpha(x_{n}, x_{n+1})d( x_{n}, x_{n+1})+\alpha(x_{n+1}, x_{n+2})d( x_{n+1}, x_{n+2})
 + \alpha(x_{n+2}, x_{n+2l+1}) \alpha( x_{n+2}, x_{n+3}))d( x_{n+2}, x_{n+3})\\
& + \alpha(x_{n+2}, x_{n+2l+1})\alpha( x_{n+3}, x_{n+4}))d( x_{n+3}, x_{n+4})+ \alpha(x_{n+2}, x_{n+2l+1})\alpha( x_{n+4}, x_{n+2k+1}))d( x_{n+4}, x_{n+2l+1})\\ 
&\leq \alpha(x_{n}, x_{n+1})d( x_{n}, x_{n+1})+\alpha(x_{n+1}, x_{n+2})d( x_{n+1}, x_{n+2})+
\alpha(x_{n+2}, x_{n+2l+1})\alpha( x_{n+2}, x_{n+3}))d( x_{n+2}, x_{n+3})\\ 
&+\alpha(x_{n+2}, x_{n+2l+1})\alpha( x_{n+3}, x_{n+4}))d( x_{n+3}, x_{n+4})\\
&+\alpha(x_{n+2}, x_{n+2l+1})\alpha( x_{n+4}, x_{n+2l+1}))\alpha(x_{n+4}, x_{n+5})d( x_{n+4}, x_{n+5})\\ 
&+\alpha(x_{n+2}, x_{n+2l+1})\alpha( x_{n+4}, x_{n+2l+1}))\alpha(x_{n+5}, x_{n+6})d( x_{n+5}, x_{n+6})\\ 
&+\alpha(x_{n+2}, x_{n+2l+1})\alpha( x_{n+4}, x_{n+2l+1}))\alpha(x_{n+6}, x_{n+2l+1})d( x_{n+6}, x_{n+2l+1})\\
&\leq \alpha(x_{n}, x_{n+1})d( x_{n}, x_{n+1})+\alpha(x_{n+1}, x_{n+2})d( x_{n+1}, x_{n+2})
+\alpha(x_{n+2}, x_{n+2l+1})\alpha( x_{n+2}, x_{n+3}))d( x_{n+2}, x_{n+3})\\ 
&+\alpha(x_{n+2}, x_{n+2l+1})\alpha( x_{n+3}, x_{n+4}))d( x_{n+3}, x_{n+4})\\
&+\alpha(x_{n+2}, x_{n+2l+1})\alpha( x_{n+4}, x_{n+2l+1}))\alpha(x_{n+4}, x_{n+5})d( x_{n+4}, x_{n+5})\\ 
&+\alpha(x_{n+2}, x_{n+2l+1})\alpha( x_{n+4}, x_{n+2l+1}))\alpha(x_{n+5}, x_{n+6})d( x_{n+5}, x_{n+6})\\
&+...+\alpha(x_{n+2}, x_{n+2l+1})\alpha( x_{n+4}, x_{n+2l+1}))\times...\times \alpha(x_{n+2l-2}, x_{n+2l+1})\times\\
&\left[\alpha(x_{n+2l-2}, x_{n+2l-1})d(x_{n+2l-2}, x_{n+2l-1})+\alpha(x_{n+2l-1}, x_{n+2l})d(x_{n+2l-1}, x_{n+2l})\right]\\
&+\alpha(x_{n+2}, x_{n+2l+1})\times \alpha(x_{n+4}, x_{n+2l+1})\times...\times \alpha(x_{n+2l-2}, x_{n+2l+1})\alpha(x_{n+2l}, x_{n+2l+1})d(x_{n+2l}, x_{n+2l+1})
\end{align*}
Therefore,
\begin{align*}
d(x_{n}, x_{m})\\
&\leq \alpha(x_{n}, x_{n+1})d_{n}+\alpha(x_{n+1}, x_{n+2})d_{n+1}+\alpha(x_{n+2}, x_{n+2l+1})\alpha( x_{n+2}, x_{n+3}))d_{n+2}\\
&+\alpha(x_{n+2}, x_{n+2l+1})\alpha( x_{n+3}, x_{n+4}))d_{n+3}
+ \alpha(x_{n+2}, x_{n+2l+1})\alpha( x_{n+4}, x_{n+2l+1})\alpha(x_{n+4}, x_{n+5})d_{n+4}\\
&+ \alpha(x_{n+2}, x_{n+2l+1})\alpha( x_{n+4}, x_{n+2l+1}))\alpha(x_{n+5}, x_{n+6})d_{n+5}\\
&+. .  . \\
&+\alpha(x_{n+2}, x_{n+2l+1})\times...\times\alpha(x_{n+2l-2}, x_{n+2l+1})\left[\alpha( x_{n+2l-2}, x_{n+2l+1}))d_{n+2l-2}+\alpha( x_{n+2l-1}, x_{n+2l}))d_{n+2l-1} \right]\\
& +\alpha(x_{n+2}, x_{n+2l+1})\alpha( x_{n+4}, x_{n+2l+1}))\times...\times\alpha(x_{n+2k-2}, x_{n+2l+1})\left[\alpha( x_{n+2l}, x_{n+2l+1}))d_{n+2l}\right]\\  
&\leq  \alpha(x_{n},x_{n+1})d(x_0,x_1) k^{n}+\alpha(x_{n+1}, x_{n+2})d(x_0,x_1)k^{n+1}\\
& +\alpha(x_{n+2}, x_{n+2l+1})\left[ \alpha( x_{n+2}, x_{n+3}))k^{n+2}+\alpha( x_{n+3}, x_{n+4}))k^{n+3}\right] d(x_0,x_1)\\ 
&+. .  .\\
&+\alpha(x_{n+2}, x_{n+2l+1})\times...\times\alpha(x_{n+2l-2}, x_{n+2l+1})\times\\
&\left[\alpha( x_{n+2l-2}, x_{n+2l-1}))k^{n+2l-2}+\alpha( x_{n+2l-1}, x_{n+2l}))k^{n+2l-2} \right]d(x_0,x_1)\\
&+\alpha(x_{n+2}, x_{n+2l+1})\alpha( x_{n+4}, x_{n+2l+1})\times...\times\alpha(x_{n+2l-2}, x_{n+2l+1})\alpha( x_{n+2l}, x_{n+2l+1}))k^{n+2l}d(x_0,x_1)\\  
&\leq  \alpha(x_{n},x_{n+1})d(x_0,x_1) k^{n}+\alpha(x_{n+1}, x_{n+2})d(x_0,x_1)k^{n+1}\\
&+\alpha(x_{n+2}, x_{n+2k+1})\left[ \alpha( x_{n+2}, x_{n+3}))k^{n+2}+\alpha( x_{n+3}, x_{n+4}))k^{n+3}\right] d(x_0,x_1)\\ 
&+. .  . \\
&+\alpha(x_{n+2}, x_{n+2l+1})\alpha(x_{n+4}, x_{n+2l+1})\times...\times\alpha(x_{n+2l-2}, x_{n+2l+1})\times\\
&\left[\alpha( x_{n+2l-2}, x_{n+2l-1}))k^{n+2l-2}+\alpha( x_{n+2l-1}, x_{n+2l}))k^{n+2l-2} \right]d(x_0,x_1)\\
&+\alpha(x_{n+2}, x_{n+2l+1})\alpha( x_{n+4}, x_{n+2l+1}))\times...\times 
\alpha(x_{n+2l-2}, x_{n+2l+1})\alpha(x_{n+2l}, x_{n+2l+1})\times\\
&\left[\alpha( x_{n+2l}, x_{n+2l+1})k^{n+2l}d(x_0,x_1)+\alpha( x_{n+2l+1}, x_{n+2l+2})k^{n+2l+1}d(x_0,x_1) \right]\\   
&\leq \alpha(x_{n},x_{n+1})d(x_0,x_1) k^{n}+\alpha(x_{n+1}, x_{n+2})d(x_0,x_1)k^{n+1}\\
&+\sum\limits_{\substack{i=n+2 }}^{i=n+2l}\prod_{ {j=n+2} }^{j=i} \alpha(x_{j},x_{n+2l+1})\left[ \alpha(x_{i},x_{i+1})k^{i}+\alpha(x_{i+1},x_{i+2})k^{i+1}\right]d(x_0,x_1). 
\end{align*}
Above, we make use of that $ \alpha(x,y)\geq 1 $. 
\begin{equation*}
Let\ S_{p}=\sum\limits_{\substack{i=0 }}^{i=p}\prod_{ {j=0} }^{j=i} \alpha(x_{j},x_{n+2l+1})\left[ \alpha(x_{i},x_{i+1})k^{i}+\alpha(x_{i+1},x_{i+1})k^{i+1}\right]d(x_0,x_1).  
\end{equation*}
Hence, we have
\begin{equation*} 
d(x_{n},x_{m})\leq d(x_0,x_1)\left[ \alpha(x_{n},x_{n+1}) k^{n}+\alpha(x_{n+1}, x_{n+2})k^{n+1}+ S_{n+m-1}- S_{n+1}\right]. 
\end{equation*}
Now, let the term $ a_i $ $ =\prod_{ {j=0} }^{j=i} \alpha(x_{j},x_{m})\left[ \alpha(x_{i},x_{i+1})k^{i}+\alpha(x_{i+1},x_{i+2})k^{i+1}\right]. $
On the other hand 
\begin{equation*}
\sup_{m\geq 1}\lim_{i\rightarrow \infty }\frac{a_{i+1}}{a_i}=\sup_{m\geq 1}\lim_{i\rightarrow \infty }\alpha\left( x_{i+1}, x_{m}\right) \frac{\alpha\left( x_{i+1}, x_{i+2}\right)k+\alpha\left( x_{i+2}, x_{i+3}\right)k^2}{\alpha\left( x_{i}, x_{i+1}\right)+\alpha\left( x_{i+1}, x_{i+2}\right)k}
\end{equation*}
\begin{equation*}
\leq \sup_{m\geq 1}\lim_{i\rightarrow \infty }\alpha\left( x_{i+1}, x_{m}\right) \frac{\alpha\left( x_{i+1}, x_{i+2}\right)+\alpha\left( x_{i+2}, x_{i+3}\right)}{\alpha\left( x_{i}, x_{i+1}\right)+\alpha\left( x_{i+1}, x_{i+2}\right)}<\frac{1}{k^2}.
\end{equation*}
 Thus the series $ \sum\limits_{\substack{i=n+2 }}^{i=\infty}\prod_{ {j=n+2} }^{j=i} \alpha(x_{j},x_{n+2l+1})\left[ \alpha(x_{i},x_{i+1})k^{i}+\alpha(x_{2i+1},x_{i+1})k^{i+1}\right]d(x_0,x_1) $ is converges.
On the other hand 
$$\lim_{n\rightarrow \infty } \alpha(x_{n},x_{n+1})d(x_0,x_1) k^{n}=\lim_{n\rightarrow \infty }\alpha(x_{n+1}, x_{n+2})d(x_0,x_1)k^{n+1}=0. $$
We conclude that
$$\lim_{n,m\rightarrow \infty }d(x_{n},x_{m})=0.  $$ 
Case $ 2 $: $ m=n+2l $ with Similarly to case $ 1 $ we have
\begin{equation*} 
d( x_{n},x_{n+2l})\leq \alpha(x_{n}, x_{n+2})d( x_{n}, x_{n+2})+\alpha(x_{n+2}, x_{n+3})d( x_{n+2}, x_{n+3})+\alpha(x_{n+3}, x_{n+2k})d( x_{n+3}, x_{n+2l+1})
\end{equation*}
\begin{equation*} 
\leq \alpha(x_{n}, x_{n+2})d( x_{n}, x_{n+2})+\alpha(x_{n+2}, x_{n+3})d( x_{n+2}, x_{n+3})+
\end{equation*}
\begin{equation*}
\alpha(x_{n+3}, x_{n+2l})\left[ \alpha( x_{n+3}, x_{n+4}))d( x_{n+3}, x_{n+4})+
\alpha( x_{n+4}, x_{n+5}))d( x_{n+4}, x_{n+5})+\alpha( x_{n+5}, x_{n+2l}))d( x_{n+5}, x_{n+2l})\right] 
\end{equation*}
\begin{equation*} 
\leq \alpha(x_{n}, x_{n+2})d( x_{n}, x_{n+2})+\alpha(x_{n+2}, x_{n+3})d( x_{n+2}, x_{n+3})+
\end{equation*}
\begin{equation*} 
\alpha(x_{n+3}, x_{n+2l})\alpha( x_{n+3}, x_{n+4}))d( x_{n+3}, x_{n+4})+\alpha(x_{n+3}, x_{n+2l})\alpha( x_{n+4}, x_{n+5}))d( x_{n+4}, x_{n+5})
\end{equation*}
\begin{equation*} 
+\alpha(x_{n+3}, x_{n+2l})\alpha( x_{n+5}, x_{n+2l}))d( x_{n+5}, x_{n+2l}) 
\end{equation*}
\begin{equation*} 
\leq \alpha(x_{n}, x_{n+2})d( x_{n}, x_{n+2})+\alpha(x_{n+2}, x_{n+3})d( x_{n+2}, x_{n+3})+
\end{equation*}
\begin{equation*} 
\alpha(x_{n+3}, x_{n+2l})\alpha( x_{n+3}, x_{n+4}))d( x_{n+3}, x_{n+4})+
\end{equation*}
\begin{equation*} 
+\alpha(x_{n+3}, x_{n+2l})\alpha( x_{n+5}, x_{n+2l}))\alpha(x_{n+5}, x_{n+6})d( x_{n+5}, x_{n+6})
\end{equation*}
\begin{equation*}
+. .  . +\alpha(x_{n+3}, x_{n+2l})\alpha( x_{n+5}, x_{n+2l}))\times...\times\alpha(x_{n+2l-3}, x_{n+2l})
\end{equation*}
\begin{equation*}
\left[ \alpha(x_{n+2l-3}, x_{n+2l-2})d( x_{n+2l-3}, x_{n+2l-2})+\alpha(x_{n+2l-2}, x_{n+2l-1})d( x_{n+2l-2}, x_{n+2l-1})\right] 
\end{equation*}
\begin{equation*}
 +\alpha(x_{n+3}, x_{n+2l})\alpha( x_{n+5}, x_{n+2l}))\times...\times\alpha(x_{n+2l-3}, x_{n+2l})\alpha(x_{n+2l-1}, x_{n+2l})d(x_{n+2l-1}, x_{n+2l})
\end{equation*}
 \begin{equation*} 
\leq \alpha(x_{n}, x_{n+2})k^{n}d( x_{0}, x_{2})+\alpha(x_{n+2}, x_{n+3})k^{n+2}d( x_{0}, x_{1})+
\end{equation*}
\begin{equation*} 
\alpha(x_{n+3}, x_{n+2l})\alpha( x_{n+3}, x_{n+4})))k^{n+3}d( x_{0}, x_{1})+
\end{equation*}
\begin{equation*}
+. .  . +\alpha(x_{n+3}, x_{n+2l})\alpha( x_{n+5}, x_{n+2k}))\times...\times\alpha(x_{n+2l-3}, x_{n+2l})
\end{equation*}
\begin{equation*}
\left[ \alpha(x_{n+2l-3}, x_{n+2l-2}))k^{n+l-3}d( x_{0}, x_{1})+\alpha(x_{n+2l-2}, x_{n+2k-1}))k^{n+l-2}d( x_{0}, x_{1}))\right] 
\end{equation*}
\begin{equation*}
 +\alpha(x_{n+3}, x_{n+2l})\alpha( x_{n+5}, x_{n+2l}))\times...\times\alpha(x_{n+2l-3}, x_{n+2l})\alpha(x_{n+2l-1}, x_{n+2l})
\end{equation*}
\begin{equation*}
\left[ \alpha(x_{n+2l-1}, x_{n+2l})k^{n+l-1}d( x_{0}, x_{1})+\alpha(x_{n+2l}, x_{n+2l})k^{n+l}d( x_{0}, x_{1})\right] 
\end{equation*}
Thus, we conclude
\begin{equation*} 
d(x_{n},x_{m})\leq \alpha(x_{n}, x_{n+2})k^{n}d( x_{0}, x_{2})+\alpha(x_{n+2}, x_{n+3})k^{n+2}d( x_{0}, x_{1})+
\end{equation*}
\begin{equation*}
+\sum\limits_{\substack{i=n+3 }}^{i=n+2l-1}\prod_{ {j=n+3} }^{j=i} \alpha(x_{j},x_{n+2l})\left[ \alpha(x_{i},x_{n+2l})k^{i}+\alpha(x_{i+1},x_{i+2})k^{i+2}\right]d(x_0,x_1)  
\end{equation*}
Above, we make use of that $ \alpha(x,y)\geq 1 $. 
\begin{equation*}
Let\ S_{q}=\sum\limits_{\substack{i=0 }}^{i=q}\prod_{ {j=0} }^{j=i} \alpha(x_{j},x_{n+2l})\left[ \alpha(x_{i},x_{i+1})k^{i}+\alpha(x_{i+1},x_{i+2})k^{i+1}\right]d(x_0,x_1).  
\end{equation*}
Hence, we have
\begin{equation*} 
d({x_n},{x_m})\leq d(x_0,x_2)\alpha(x_{n},x_{n+2}) k^{n}+ d(x_0,x_1)\left[\alpha(x_{n+2}, x_{n+3})k^{n+2}+ S_{m-1}- S_{n-2}\right] 
\end{equation*}
On the other hand 
\begin{align*}
\sup_{m\geq 1}\lim_{i\rightarrow \infty }\frac{a_{i+1}}{a_i}&=\sup_{m\geq 1}\lim_{i\rightarrow \infty }\alpha\left( x_{i+1}, x_{m}\right) \frac{\alpha\left( x_{i+1}, x_{i+2}\right)k+\alpha\left( x_{i+2}, x_{i+3}\right)k^2}{\alpha\left( x_{i}, x_{i+1}\right)+\alpha\left( x_{i+1}, x_{i+2}\right)k}\\
&\leq\sup_{m\geq 1}\lim_{i\rightarrow \infty }\alpha\left( x_{i+1}, x_{m}\right) \frac{\alpha\left( x_{i+1}, x_{i+2}\right)+\alpha\left( x_{i+2}, x_{i+3}\right)}{\alpha\left( x_{i}, x_{i+1}\right)+\alpha\left( x_{i+1}, x_{i+2}\right)}<\frac{1}{k^2}
\end{align*}
By using the Ratio Test, it is not difficult to that the series 
$$ \sum\limits_{\substack{i=0 }}^{i=\infty}\prod_{ {j=0} }^{j=i} \alpha(x_{j},x_{n+2l+1})\left[ \alpha(x_{i},x_{i+1})k^{i}+\alpha(x_{i+1},x_{i+1})k^{i+1}\right]d(x_0,x_1)$$ 
converges. Hence $ d(x_{n},x_{m}) $ is converges as $ n,m $ go toward $ \infty $ . Thus, by case1 and case 2, we have
\begin{equation}
 \lim_{n,m\rightarrow \infty }d(x{_n},x_{m})=0.
\end{equation}
We conclude that the sequence $ x_{n} $  is a Cauchy sequence in the complete controlled rectangular metric space $\left(X,d\right)  $, so $ x_{n} $ converges to some $ z\in X $.

We shall show that $ z $ is a fixed point of $ T. $

No, we show that $ d(Tz,z)=0 $. Arguing by contradiction, we assume that $ d(Tz,z)>0 $.
From  the controlled rectangular inequality we get, 
\begin{equation}
d\left( z,Tz\right) \leq \alpha(z,x_{n})d\left(z,x_{n}\right) +\alpha(x_{n},Tx_{n})d\left(
x_{n},Tx_{n}\right) +\alpha(Tx_{n},Tz)d\left( Tx_{n},Tz\right). 
\end{equation}
From assumption of the hypothesis, we have
\begin{equation}
d\left( z,Tz\right) \leq \alpha(z,x_{n})d\left(z,x_{n}\right) +\alpha(x_{n},Tx_{n})d\left(
x_{n},Tx_{n}\right) +k\alpha(Tx_{n},Tz)d\left( x_{n},z\right). 
\end{equation}
Therefore,
\begin{equation}
d\left( z,Tz\right)\leq d\left(z,x_{n}\right)\left[ \alpha(z,x_{n})+k\alpha(Tx_{n},Tz)\right]+\alpha(x_{n},Tx_{n})d\left(
x_{n},Tx_{n}\right). 
\end{equation}
By letting ${n\rightarrow \infty }$ above inequality and using $ (3.3) $, we deduce that $ d\left( z,Tz\right)=0 $ 
and that is $ Tz=0 $. Thus $ T $ has a fixed point $ z\in X $.

Uniqueness: assume there exist two fixed points of $ T $ say $ z $ and $ u $ such that $ z\neq u $.
By the contractive property of $ T $ we have 
\begin{equation*}
d(z,u)=d(Tz,Tu)\leq kd(z,u).
\end{equation*}
Which is a contradiction. Thus, $ T $ has a unique fixed point.
\end{proof}
 \begin{theorem}
Let $ (X,d) $ be a complete controlled rectangular metric space, and $ T $ a self mapping on $ X $ satisfying the following condition:
for all $ x,y\in X $ there exists $ 0 < k <\frac{1}{2} $ such that
\begin{equation}
 d(Tx,Ty)\leq k\left[ d(x,Tx)+d(y,Ty)\right].
\end{equation}
Let $ x_0\in X $, take $ x_{n}=T^{n}x_{0} $ Also, if 
\begin{equation}
\sup_{m\geq 1}\lim_{i\rightarrow \infty }\alpha\left( x_{i+1}, x_{m}\right) \frac{\alpha\left( x_{i+1}, x_{i+2}\right)+\alpha\left( x_{i+2}, x_{i+3}\right)}{\alpha\left( x_{i}, x_{i+1}\right)+\alpha\left( x_{i+1}, x_{i+2}\right)}<\frac{1}{k^2}. 
\end{equation}
We assume that
$ \lim\limits_{n \rightarrow +\infty}\alpha(x_{n},x),\  \lim\limits_{n \rightarrow +\infty}\alpha(x,x_{n}) \ and  \  \lim\limits_{n,m \rightarrow +\infty}\alpha(x_{n},x_{m}), $
 exist  and  are finite for all $n,m\in\mathbb{N},\ n\neq m$
Such that
 \begin{equation}
 \lim\limits_{n \rightarrow +\infty}\alpha(Tx_{n},Tx)<\frac{1}{k}.
\end{equation}
Then $ T $ has a unique fixed point in X.
\end{theorem}
\begin{proof}
Let $ x_0\in X $ and define the sequence $ {x_n} $ as follows 
$ x_1=T{x_0},\  x_2=T{x_1},...,\ x_n=Tx_{n-1}=T^n{x_0},...$ 
Now we prove that 
\begin{equation*}
\lim_{n\rightarrow \infty}d(x_{n},x_{n+1})=0
\end{equation*}
and
\begin{equation*} 
\lim_{n\rightarrow \infty}d(x_{n},x_{n+2})=0.
\end{equation*}
Using the contractive property with $ x=x_{n-1} $ and $ y=x_{n} $, we obtain
\begin{equation*}
d(x_{n},x_{n+1})=d(Tx_{n-1},Tx_{n})\leq k\left[ d(x_{n-1},x_{n})+d(x_{n},x_{n+1})\right]
\end{equation*}
\begin{equation*}
\Rightarrow d(x_{n},x_{n+1})\leq \frac{k}{1-k}  d(x_{n-1},x_{n}).
\end{equation*}
Since $ 0 < k < \frac{1}{2}$, one can easily deduce that  $ 0<\frac{k}{1-k}<1$. So, let $ \beta= \frac{k}{1-k}$
hence, 
\begin{align*}
    d(x_{n},x_{n+1})      &  \leq \beta d(x_{n-1},x_{n})   \\
                          & \leq \beta^2 d(x_{n-2},x_{n-1})\\
                          & \leq...\\
                          & \leq \beta^{n} d(x_{0},x_{1}).\\
\end{align*}
Therefore,
\begin{equation} 
\lim_{n\rightarrow \infty}d(x_{n},x_{n+1})=0.
\end{equation}
Appliyin $ (2.10) $ with $ x=x_{n-1} $ and $ y= x_{n+1}$, we obtain
\begin{equation*}
d(x_{n},x_{n+2})=d(Tx_{n-1},Tx_{n+1})\leq k\left[ d(x_{n-1},x_{n})+d(x_{n+1},x_{n+2})\right].
\end{equation*} 
Thus, by using the fact that $ d(x_{n},x_{n+1})\rightarrow 0$ as $ n\rightarrow \infty$, we deduce that
\begin{equation} 
\lim_{n\rightarrow \infty}d(x_{n},x_{n+2})=0.
\end{equation}
Now, similarly to prove of Theorem $ (2.7) $, we deduce that the sequence $ \lbrace x_n\rbrace $ is a Cauchy sequence in $ X $. Since $ (X,d) $ is a complete controlled rectangular metric space, we conclude that $ x_n $ converges to some $ z $ in $ X. $

We shall show that $ z $ is a fixed point of $ T. $

No, we show that $ d(Tz,z)=0 $. From  the controlled rectangular inequality we get, 
\begin{equation}
d\left( z,Tz\right) \leq \alpha(z,x_{n})d\left(z,x_{n}\right) +\alpha(x_{n},Tx_{n})d\left(
x_{n},Tx_{n}\right) +\alpha(Tx_{n},Tz)d\left( Tx_{n},Tz\right). 
\end{equation}
From assumption of the hypothesis, we have
\begin{equation}
d\left( z,Tz\right) \leq \alpha(z,x_{n})d\left(z,x_{n}\right) +\alpha(x_{n},Tx_{n})d\left(
x_{n},Tx_{n}\right) +k \alpha(Tx_{n},Tz)\left[ d\left( x_{n},Tx_{n}\right)+d\left( z,Tz\right)\right].
\end{equation}
Therefore,
\begin{equation}
d\left( z,Tz\right) \leq \frac{1}{1	-k\alpha(Tx_{n},Tz)}\left[ \alpha(z,x_{n})d\left(z,x_{n}\right) +\alpha(x_{n},Tx_{n})d\left(
x_{n},Tx_{n}\right) + \alpha(Tx_{n},Tz) d\left( x_{n},Tx_{n}\right)\right] .
\end{equation}
Letting $n\rightarrow \infty $ in $\left( 2.17\right) $ and using $ (2.12) $, we obtain
\begin{equation*}
d\left( z,Tz\right) \leq0 .
\end{equation*}
 Which is a contradiction. Thus, $z= Tz $.
 
Uniqueness: assume there exist two fixed points of $ T $ say $ z $ and $ u $ such that $ z\neq u $.
By the contractive property of $ T $ we have 
\begin{equation*}
d(z,u)=d(Tz,Tu)\leq k\left[ d(z,Tz)+d(u,Tu)\right]=0.
\end{equation*}
Hence $ z=u $.
\end{proof}
\begin{theorem}
Let $ (X,d) $ be a complete controlled rectangular metric space, and $ T $ a self mapping on $ X $ satisfying the following condition:
for all $ x,y\in X $ there exists $ 0 < \lambda <\frac{1}{3} $ such that
\begin{equation}
d(Tx,Ty)\leq \lambda\left[d(x,y)+ d(x,Tx)+d(y,Ty)\right].
\end{equation}
Let $ x_0\in X $, take $ x_{n}=T^{n}x_{0} $ Also, if 
\begin{equation}
\sup_{m\geq 1}\lim_{i\rightarrow \infty }\alpha\left( x_{i}, x_{m}\right) \frac{\alpha\left( x_{i+1}, x_{i+2}\right)+\alpha\left( x_{i+2}, x_{i+3}\right)}{\alpha\left( x_{i}, x_{i+1}\right)+\alpha\left( x_{i+1}, x_{i+2}\right)}<\frac{1}{k^2}.
\end{equation}
We assume that
$ \lim\limits_{n \rightarrow +\infty}\alpha(x_{n},x),\  \lim\limits_{n \rightarrow +\infty}\alpha(x,x_{n}) \ and  \  \lim\limits_{n,m \rightarrow +\infty}\alpha(x_{n},x_{m}), $
 exist  and  are finite for all $n,m\in\mathbb{N},\ n\neq m$
Such that
 \begin{equation}
  \lim\limits_{n \rightarrow +\infty}\alpha(x_{n},T^{2}x_{n})<\frac{1}{\lambda}, \\
    \lim\limits_{n \rightarrow +\infty}\alpha(Tx_{n},Tx)<\frac{1}{\lambda}\ and \ \lim\limits_{n \rightarrow +\infty}\alpha(Tx,Tx_{n})<\frac{1}{\lambda}.
\end{equation}
Then $ T $ has a unique fixed point in X.
\end{theorem}
\begin{proof}
Let $ x_0\in X $ and define the sequence $ {x_n} $ as follows 
$ x_1=T{x_0},\  x_2=T{x_1},...,\ x_n=Tx_{n-1}=T^n{x_0},...$ 
Now we prove that 
\begin{equation*}
\lim_{n\rightarrow \infty}d(x_{n},x_{n+1})=0
\end{equation*}
and\begin{equation*} 
\lim_{n\rightarrow \infty}d(x_{n},x_{n+2})=0.
\end{equation*}
Using the contractive property, we have
\begin{align*}
    d(x_{n},x_{n+1})=d(Tx_{n-1},Tx_{n}) &  \leq \lambda\left[ d(x_{n-1},x_{n})+d(x_{n-1},x_{n})+d(x_{n},x_{n+1})\right]. \\
\end{align*}
So, we have
\begin{align*}
d(x_{n},x_{n+1})\leq \frac{2\lambda}{1-\lambda}  d(x_{n-1},x_{n}).
\end{align*}
Since $ 0 <\lambda < \frac{1}{3}$, one can easily deduce that  $ 0< \frac{2\lambda}{1-\lambda}<1$. So, let $ \beta= \frac{2\lambda}{1-\lambda}.$\\
Hence, 
\begin{align*}
    d(x_{n},x_{n+1})      &  \leq \beta d(x_{n-1},x_{n})   \\
                          & \leq \beta^2 d(x_{n-2},x_{n-1})\\
                          & \leq...\\
                          & \leq\beta^{n} d(x_{0},x_{1}).\\
\end{align*}
Therefore,
\begin{equation} 
\lim_{n\rightarrow \infty}d(x_{n},x_{n+1})=0.
\end{equation}
Applying $ (2.18) $ with $ x=x_{n-1} $ and $ y=x_{n+1} $, we obtain 
\begin{align*}
d(x_{n},x_{n+2})=d(Tx_{n-1},Tx_{n+1})&\leq\lambda\left[ d(x_{n-1},x_{n})+d(x_{n+1},x_{n+2}+d(x_{n-1},x_{n+1})\right] 
\end{align*}
Using the property $ (3)$ of the controlled rectangular metric space we get, 
\begin{align*}
d(x_{n},x_{n+2}) &  \leq\lambda\left[ d(x_{n-1},x_{n+1})+d(x_{n-1},x_{n})+d(x_{n+1},x_{n+2})\right]\\
&\leq \lambda \left( d(x_{n-1},x_{n})+d(x_{n+1},x_{n+2})\right)\\
& +\lambda \left[\alpha(x_{n-1},x_{n+2}) d(x_{n-1},x_{n+2})+\alpha(x_{n+2},x_{n}) d(x_{n+2},x_{n})+\alpha(x_{n},x_{n+1}) d(x_{n},x_{n+1})\right]\\
&\leq \lambda \left( d(x_{n-1},x_{n})+d(x_{n+1},x_{n+2})\right)\\
&+ \lambda \left[\alpha(x_{n+2},x_{n}) d(x_{n+2},x_{n})+\alpha(x_{n},x_{n+1}) d(x_{n},x_{n+1})\right]\\
& +\lambda \alpha(x_{n-1},x_{n+2})\left[ \alpha(x_{n-1},x_{n})d(x_{n-1},x_{n})+\alpha(x_{n},x_{n+1})d(x_{n},x_{n+1})+ \alpha(x_{n+1},x_{n+2})d(x_{n+1},x_{n+2}))\right]
\end{align*}
Therefore, we have
\begin{equation}
d(x_{n},x_{n+2})\leq d(x_{n-1},x_{n}) \left[ \frac{2\lambda+\lambda \alpha(x_{n},x_{n+1})+\lambda \alpha(x_{n-1},x_{n+2})\left[  \alpha(x_{n-1},x_{n}))  +\alpha(x_{n},x_{n+1})+\alpha(x_{n+1},x_{n+2})\right]  }{1-\lambda \alpha(x_{n},x_{n+2})}\right]  
\end{equation}
Letting $n\rightarrow \infty $ in $\left( 2.22\right) $, using $ (2.20 )$, and $(2.21 )$, we obtain
\begin{equation} 
\lim_{i\rightarrow \infty}d(x_{n},x_{n+2})=0.
\end{equation}
Now, similarly to prove of Theorem $ (2.7) $, we deduce that the sequence $\lbrace x_n\rbrace $ is a Cauchy sequence. Since $ (X,d) $ is a complete controlled rectangular metric space, we conclude that $ x_n $ converges to some $ z $ in $ X. $

We shall show that $ z $ is a fixed point of $ T.$
No, we show that $ d(Tz,z)=0 $. From  the controlled rectangular inequality we get, 
\begin{equation}
d\left( z,Tz\right) \leq \alpha(z,x_{n})d\left(z,x_{n}\right) +\alpha(x_{n},Tx_{n})d\left(
x_{n},Tx_{n}\right) +\alpha(Tx_{n},Tz)d\left( Tx_{n},Tz\right).
\end{equation}
From assumption of the hypothesis, we have
\begin{equation}
d\left( z,Tz\right) \leq \alpha(z,x_{n})d\left(z,x_{n}\right) +\alpha(x_{n},Tx_{n})d\left(
x_{n},Tx_{n}\right) +\lambda \alpha(Tx_{n},Tz)\left[ d\left( x_{n},Tx_{n}\right)+d\left( z,Tz\right)+d\left( x_{n},z\right)\right].
\end{equation}
Letting $n\rightarrow \infty $ in $\left( 2.25\right) $, we obtain
\begin{align*}
d\left( z,Tz\right)& \leq \lambda \lim_{n\rightarrow \infty }\alpha(Tx_{n},Tz)\left[ d\left( z,Tz\right)\right]< d\left( z,Tz\right). 
\end{align*}
Which is a contradiction. Thus, $z= Tz $.

Uniqueness: assume there exist two fixed points of $ T $ say $ z $ and $ u $ such that $ z\neq u $.
By the contractive property of $ T $ we have 
\begin{equation*}
d(z,u)=d(Tz,Tu)\leq\lambda\left[d(z,u)+ d(z,Tz)+d(u,Tu)\right]
= \lambda d(z,Tz)<d(z,u).
\end{equation*}
Hence $ z=u $.
\end{proof}
In the following we prove some new fixed point result for rational contraction of Fisher \cite{FICH} type in the context of controlled rectangular metric space. 
 \begin{theorem}
Let $ (X,d) $ be a complete controlled rectangular metric space, and $ T $ a self mapping on $ X $ satisfying the following condition:
for all $ x,y\in X $ there exists $ \lambda,\beta \in\left]0,1 \right[ $ where  $  \lambda+\beta <1 $. For $ x_{0}\in X$, take $ x_{n}=T^{n}x_0 $. Such that
\begin{equation}
 d(Tx,Ty)\leq \lambda d(x,y)+ \beta \frac{ d(x,Tx)+d(y,Ty)}{1+ d(x,y)} .
\end{equation}
Also, if 
\begin{equation}
\sup_{m\geq 1}\lim_{i\rightarrow \infty }\alpha\left( x_{i+1}, x_{m}\right) \frac{\alpha\left( x_{i+1}, x_{i+2}\right)+\alpha\left( x_{i+2}, x_{i+3}\right)}{\alpha\left( x_{i}, x_{i+1}\right)+\alpha\left( x_{i+1}, x_{i+2}\right)}<\frac{1}{(\beta+\lambda)^{2}}. 
\end{equation}
We assume that
$ \lim\limits_{n \rightarrow +\infty}\alpha(x_{n},x),\  \lim\limits_{n \rightarrow +\infty}\alpha(x,x_{n}) \ and  \  \lim\limits_{n,m \rightarrow +\infty}\alpha(x_{n},x_{m}), $
 exist  and  are finite for all $n,m\in\mathbb{N},\ n\neq m$
Such that
 \begin{equation}
  \lim\limits_{n,m \rightarrow +\infty}\alpha(x_{n},x_{m})<\frac{1}{(\beta+\lambda)}, \\
    \lim\limits_{n \rightarrow +\infty}\alpha(x_{n},x)<\frac{1}{(\beta+\lambda)}\ and \ \lim\limits_{n \rightarrow +\infty}\alpha(x_{n},x)<\frac{1}{(\beta+\lambda)}.
\end{equation}
Then $ T $ has a unique fixed point in X.
\end{theorem}
\begin{proof}
Let $ x_0\in X $ and define the sequence $ {x_n} $ as follows 
$ x_1=T{x_0},\  x_2=T{x_1},...,\ x_n=Tx_{n-1}=T^n{x_0},...$ 
Now we prove that 
\begin{equation*}
\lim_{n\rightarrow \infty}d(x_{n},x_{n+1})=0
\end{equation*}
and
\begin{equation*} 
\lim_{n\rightarrow \infty}d(x_{n},x_{n+2})=0.
\end{equation*}
Using the contractive property we have
\begin{align*}
d(x_{n},x_{n+1})&=d(Tx_{n-1},Tx_{n})\\
&\leq \lambda d(x_{n-1},x_{n})+\beta \frac{d(x_{n-1},x_{n})d(x_{n},x_{n+1})}{1+d(x_{n-1},x_{n})}\\
&\leq \lambda d(x_{n-1},x_{n})+\beta d(x_{n},x_{n+1}),
\end{align*}
which implies
\begin{equation*} 
d(x_{n},x_{n+1})\leq \frac{\lambda}{1-\beta } d(x_{n-1},x_{n})=\gamma d(x_{n-1},x_{n}),
\end{equation*}
where $ \gamma= \frac{\lambda}{1-\beta }$, then $  \gamma \in\left]0,1 \right[ $. Thus , we have 
\begin{align*}
    d(x_{n},x_{n+1})      &  \leq  \gamma d(x_{n-1},x_{n})   \\
                          & \leq  \gamma^2 d(x_{n-2},x_{n-1})\\
                          & \leq...\\
                          & \leq  \gamma^{n} d(x_{0},x_{1}).\\
\end{align*}
Therefore,
\begin{equation} 
\lim_{n\rightarrow \infty}d(x_{n},x_{n+1})=0.
\end{equation}
Appliyin $ (2.26) $ with $ x=x_{n-1} $ and $ y= x_{n+1}$, we obtain
\begin{align*}
d(x_{n},x_{n+2})&=d(Tx_{n-1},Tx_{n+1})\\
&\leq \lambda d(x_{n-1},x_{n+1})+\beta \frac{d(x_{n-1},x_{n})d(x_{n+1},x_{n+2})}{1+d(x_{n-1},x_{n})}\\
&\leq \lambda d(x_{n-1},x_{n+1})+\beta d(x_{n},x_{n+1})\\
&\leq \beta d(x_{n},x_{n+1})+\lambda \left[ \alpha(x_{n-1},x_{n+2})d(x_{n-1},x_{n+2})+\alpha(x_{n+2},x_{n})d(x_{n+2},x_{n})\alpha(x_{n},x_{n+1})d(x_{n},x_{n+1})\right]\\
& \leq \beta d(x_{n},x_{n+1})+\lambda \left[\alpha(x_{n+2},x_{n})d(x_{n+2},x_{n})+\alpha(x_{n},x_{n+1})d(x_{n},x_{n+1})\right]\\
&+\lambda\alpha(x_{n-1},x_{n+2})\left[\alpha(x_{n-1},x_{n})d(x_{n-1},x_{n})+\alpha(x_{n},x_{n+1})d(x_{n},x_{n+1})\alpha(x_{n+1},x_{n+2})d(x_{n+1},x_{n+2})\right],
\end{align*}
which implies
\begin{equation} 
d(x_{n},x_{n+2})\leq \frac{d(x_{n-1},x_{n})}{1-\lambda \alpha(x_{n+2},x_{n})}\left[\beta+\lambda \alpha(x_{n},x_{n+1})+\alpha(x_{n-1},x_{n+2})\left(\alpha(x_{n-1},x_{n})+\alpha(x_{n},x_{n+1})+\alpha(x_{n+1},x_{n+2}) \right)  \right] .
\end{equation}
Thus, by using the fact that $ d(x_{n},x_{n+1})\rightarrow 0$ as $ n\rightarrow \infty$ and using $ (2.28) $, we deduce that
\begin{equation} 
\lim_{n\rightarrow \infty}d(x_{n},x_{n+2})=0.
\end{equation}
Now, similarly to prove of Theorem $ (2.7) $, we deduce that the sequence $ \lbrace x_n\rbrace $ is a Cauchy sequence in $ X $. Since $ (X,d) $ is a complete controlled rectangular metric space, we conclude that $ x_n $ converges to some $ z $ in $ X. $

We shall show that $ z $ is a fixed point of $ T. $
No, we show that $ d(Tz,z)=0 $. From  the controlled rectangular inequality we get, 
\begin{equation}
d\left( z,Tz\right) \leq \alpha(z,x_{n})d\left(z,x_{n}\right) +\alpha(x_{n},Tx_{n})d\left(
x_{n},Tx_{n}\right) +\alpha(Tx_{n},Tz)d\left( Tx_{n},Tz\right). 
\end{equation}
From assumption of the hypothesis, we have
\begin{align}
d\left( z,Tz\right)& \leq \alpha(z,x_{n})d\left(z,x_{n}\right) +\alpha(x_{n},Tx_{n})d\left(x_{n},Tx_{n}\right)
 +\alpha(Tx_{n},Tz)\left[ \lambda d\left( x_{n},z\right)+ \beta \frac{d\left( x_{n},Tx_{n}\right)d\left( z,Tz\right)}{1+d\left( x_{n},z\right)}\right].
\end{align}
By letting $n\rightarrow \infty $ in  $ \left( 2.33\right) $, we obtain
\begin{equation*}
d\left( z,Tz\right) \leq 0.
\end{equation*}
 Which is a contradiction. Thus, $z= Tz $.
 
Uniqueness: assume there exist two fixed points of $ T $ say $ z $ and $ u $ such that $ z\neq u $.
By the contractive property of $ T $ we have 
\begin{equation*}
d(z,u)=d(Tz,Tu)\leq \lambda  d(z,u)+\beta \frac{ d(z,Tz) d(u,Tu)}{1+ d(z,u)}= \lambda  d(z,u)< d(z,u).
\end{equation*}
Hence $ z=u $.
\end{proof}
\begin{exemple}
	Let $ X=A\cup B $, where $ A=\lbrace \frac{1}{n}:n\in\lbrace 2,3,4,5\rbrace \rbrace $ and $ B=\left[1,2 \right]  $. Define $ d:X\times X\rightarrow \left[0,+\infty \right[  $ as follows:
	\begin{equation*}
	\left\lbrace
	\begin{aligned}
	d(x, y) &=d(y, x)\ for \ all \  x,y\in X;\\
	d(x, y) &=0\Leftrightarrow y= x.\\	
	\end{aligned}
	\right.
	\end{equation*}
	and
	\begin{equation*}
	\left\lbrace
	\begin{aligned}		    
	d\left( \frac{1}{3},\frac{1}{4}\right) =d\left( \frac{1}{4},\frac{1}{5}\right) &=0,04\\
	d\left( \frac{1}{3},\frac{1}{5}\right) =d\left( \frac{1}{4},\frac{1}{6}\right) &=0,09\\
	d\left( \frac{1}{3},\frac{1}{6}\right) =d\left( \frac{1}{5},\frac{1}{6}\right)&=0,36\\
	d\left( x,y\right) =\left( \vert x-y\vert\right) ^{2} \ otherwise.
	\end{aligned}
	\right.
	\end{equation*}
	Define mapping $\alpha :X\rightarrow X$ by
	\begin{equation*}
	\alpha(x,y)=\left\lbrace
	\begin{aligned}
	\max\lbrace x,y\rbrace +2	& \ if \ x,y \in \left[1,2 \right]\\
	 3 &   \ otherwise.\\
	\end{aligned}
	\right.
	\end{equation*}
	Then $ (X,d) $ is a is controlled rectangular metric space. However we have the following:
	\begin{itemize}
	\item[1)] $ (X,d) $ is not a  metric space, as $$d\left( \frac{1}{3},\frac{1}{6}\right)=0.36>0.13=d\left( \frac{1}{3},\frac{1}{4}\right)+d\left( \frac{1}{4},\frac{1}{6}\right) 	.$$
	\item[2)] $ (X,d) $ is not a  controlled metric space, as $$d\left( \frac{1}{5},\frac{1}{6}\right)=0.36>0.3033= \alpha\left( \frac{1}{5},\frac{1}{4}\right) d\left( \frac{1}{5},\frac{1}{4}\right)+\alpha\left( \frac{1}{4},\frac{1}{6}\right) d\left( \frac{1}{4},\frac{1}{6}\right) 	.$$ 
	\item[3)] $ (X,d) $ is not a rectangular metric space, as $$d\left( \frac{1}{5},\frac{1}{6}\right)=0.36>0.22=d\left( \frac{1}{5},\frac{1}{3}\right)+d\left( \frac{1}{3},\frac{1}{4}\right)+d\left( \frac{1}{4},\frac{1}{6}\right).$$ 
\end{itemize}
	Define mapping $T:X\rightarrow X$ by
	\begin{equation*}
	T(x)=\left\lbrace
	\begin{aligned}
	x^{\frac{1}{2}}	& \ if \ x\in \left[1,2 \right]\\
	1 &  \ if \ x\in A.\\
	\end{aligned}
	\right.
	\end{equation*}
	Then, $T(x)\in \left[1,2 \right]$. Let $k=\frac{1}{2}$.
	
	Consider the following possibilities:
	
	case 1: $ x,y \in \left[1,2 \right]$ with $ x\neq y $, assume that $ x>y $.
	$$d(Tx,Ty)=\left[  x^{\frac{1}{2}}-y^{\frac{1}{2}}\right]^{2}  .$$
	and
	$$  k.d(x,y)) = \frac{1}{2}\left[x-y\right]^{2} .$$
	On the other hand 
	\begin{align*}
	 d(Tx,Ty)-  k.d(x,y))& =\left[  x^{\frac{1}{2}}-y^{\frac{1}{2}}\right]^{2}-\frac{1}{2}\left[x-y\right]^{2}\\
	&=\frac{1}{2}\left( x^{\frac{1}{2}}-y^{\frac{1}{2}} \right) \left(\sqrt{2}x^{\frac{1}{2}}-\sqrt{2}y^{\frac{1}{2}}+x-y\right)\left(\sqrt{2}-x^{\frac{1}{2}}-y^{\frac{1}{2}}\right).  
	\end{align*}
	Since $ x,y\in \left[1,2 \right]  $, then
	$$\left(\sqrt{2}-x^{\frac{1}{2}}-y^{\frac{1}{2}}\right)\leq 0  .$$
	Which implies that
	\begin{align*}
	 d(Tx,Ty) &\leq  k.d(x,y)).
	\end{align*}
	case 2: $ x\in \left[1,2 \right], y \in A  $ or $ y\in \left[1,2 \right],x \in A  $ . 
	
Therefore, $ T(x)=x^{\frac{1}{2}} $, $ T(y)=1 $, then $ d(Tx,Ty)=\left( \vert x^{\frac{1}{2}}-1 \vert \right) ^{2}=\left(  x^{\frac{1}{2}}-1  \right) ^{2}. $

Since, $ x\geq y $ for all $ x\in \left[1,2 \right], y \in A  $. Therefore,
$k.d(x,y)=\frac{1}{2}\left( x-y \right)^{2}.$

On the other hand
\begin{align*}
 0\leq \left( x^{\frac{1}{2}}-1  \right) ^{2} & \leq \left( 2^{\frac{1}{2}}-1 \right) ^{2}\\
 & \leq	\frac{2}{9}\\
 &=\frac{1}{2}\left(1-\frac{1}{3} \right)^{2}\\
 &\leq \frac{1}{2}\left(x-\frac{1}{3} \right)^{2}\\
 & \leq \frac{1}{2}\left(x-y \right)^{2}.   
\end{align*}
Which implies that
	\begin{align*}
	 d(Tx,Ty) &\leq   k.d(y,Ty)).
	\end{align*}
case 3: $ x,y \in A$ 
$$ d(Tx,Ty)=0. $$
Which implies that
	\begin{align*}
	 d(Tx,Ty) &\leq   k.d(y,Ty)).
	\end{align*}
Note that for each $x\in X  $,
\begin{equation*}
	T^{n}(x)=\left\lbrace
	\begin{aligned}
	x^{\frac{1}{2^{n}}}	& \ if \ x\in \left[1,2 \right]\\
	1 &  \ if \ x\in A.\\
	\end{aligned}
	\right.
	\end{equation*}
Thus we obtain:
\begin{equation*}
\lim_{i\rightarrow \infty }\sup_{m\geq 1}\alpha\left( x_{i}, x_{m}\right) \frac{\alpha\left( x_{i+1}, x_{i+2}\right)+\alpha\left( x_{i+2}, x_{i+3}\right)}{\alpha\left( x_{i}, x_{i+1}\right)+\alpha\left( x_{i+1}, x_{i+2}\right)}=3<4=\frac{1}{k^2}.
\end{equation*}
On the other hand
\begin{equation*}
\lim\limits_{n \rightarrow +\infty}\alpha(x_{n},x)=  \lim\limits_{n \rightarrow +\infty}\alpha(x,x_{n})\leq 4 \ and  \  \lim\limits_{n,m \rightarrow +\infty}\alpha(x_{n},x_{m})=3 \  \forall n,m\in\mathbb{N},\ n\neq m.
\end{equation*}
Therefore, all conditions of Theorem $(3.7)  $ are satisfied hence $ T $ has a unique fixed point $ Z=1. $
	\end{exemple} 
\bibliographystyle{amsplain}

\end{document}